\documentclass{amsart}
\newtheorem{theorem}{Theorem}[section]
\newtheorem{lemma}[theorem]{Lemma}
\newtheorem{corollary}[theorem]{Corollary}

\theoremstyle{definition}

\theoremstyle{remark}
\newtheorem{remark}[theorem]{Remark}
\numberwithin{equation}{section}

\newcommand{\eps}{\varepsilon}

\newcommand{\ml}{\mathcal L}

\usepackage{color,graphicx,amssymb}
\usepackage{graphicx}
\allowdisplaybreaks
\begin{document}
\title{Rigorous Pointwise approximations for invariant densities of nonuniformly expanding maps}
\author{Wael Bahsoun$^{\dagger}$}
\address{Department of Mathematical Sciences, Loughborough University,
Loughborough, Leicestershire, LE11 3TU, UK}
\email{$\dagger$ W.Bahsoun@lboro.ac.uk, $\ddagger$ Y.Duan@lboro.ac.uk}
\author{Christopher Bose$^{*}$}
\address{Department of Mathematics and Statistics, University of Victoria,
   PO BOX 3045 STN CSC, Victoria, B.C., V8W 3R4, Canada}
\email{$*$ cbose@uvic.ca}
\author{Yuejiao Duan$^{\ddagger}$}
\date{\today}
\begin{abstract}
We use an Ulam-type discretization scheme to provide {\em{pointwise}} approximations for invariant densities of interval maps with a neutral fixed point. We prove that the approximate invariant density converges pointwise to the true density at a rate $C^*\cdot\frac{\ln m}{m}$, where $C^*$ is a computable fixed constant and $m^{-1}$ is the mesh size of the discretization.
\end{abstract}
\maketitle
\pagestyle{myheadings}
\markboth{Pointwise Approximations of Invariant Densities}{W. Bahsoun, C. Bose, Y. Duan}

\section{introduction}
Ulam-type discretization schemes provide rigorous approximations for dynamical invariants. Moreover, such discretizations are easily
implementable on a computer. In \cite{Li} it was shown that the original Ulam method \cite{Ulam} is remarkably successful in approximating isolated spectrum of transfer operators associated with
piecewise expanding maps of the interval. In particular, it  was shown that this method provides rigorous approximations in the $L^{1}-$norm for invariant densities of Lasota-Yorke maps (see \cite{Li} and references therein). This method has been also successful when dealing with multi-dimensional piecewise expanding maps \cite{M2}, and partially successful\footnote{See \cite{BKL} for examples where the pure Ulam method provides fake spectra for certain hyperbolic systems.} in providing rigorous approximations for certain uniformly hyperbolic systems \cite{Froy1, Froy2}. Recently, Blank \cite{Blank} and Murray \cite{RM} independently succeeded in applying the pure Ulam method in a non-uniformly hyperbolic setting. They obtained approximations in the $L^{1}-$norm for invariant densities of certain non-uniformly expanding maps of the interval\footnote{In \cite{RM}, in addition to proving convergence, Murray also obtained an upper bound on the rate of convergence.}.

\bigskip

Although $L^1$ approximations provide significant information about the long-term statistics of the underlying system, they are not helpful when dealing with rare events in dynamical systems. In fact,  when studying \textit{rare events} in dynamical systems \cite{AFV, K} one often obtains probabilistic laws that depend on pointwise information from the invariant density of the system. In particular, \textit{extreme value laws} of interval maps with a neutral fixed point depend pointwise on the invariant density of the map \cite{HNT}.

\bigskip
Statistical properties of non-uniformly expanding maps were studied by Pianigiani in \cite{P} who first proved existence of invariant densities of such maps. Later, it was independently proved in \cite{H,LSV,Y} that such maps exhibit polynomial decay of correlations. The slow mixing behaviour that such maps exhibit has made them good testing tools for real and difficult physical problems.

\bigskip

The difficulty in obtaining pointwise approximations for invariant densities of interval maps with a neutral fixed point is two fold. Firstly, the transfer operator associated with such maps does not have a spectral gap in a classical Banach space. Therefore, powerful perturbation results \cite{KL}\footnote{See also \cite{GN} for another perturbation result.} are not directly available in this setting. Secondly, invariant densities of such maps are not $L^{\infty}$ functions. Consequently, to provide pointwise approximation of such densities, one should first measure the approximations in a `properly weighted' $L^{\infty}-$norm.

\bigskip

In this note we use a piecewise linear Ulam-type discretization scheme to provide pointwise approximations for invariant densities of nonuniformly expanding interval maps. Our main result is stated in Corollary \ref{co1}. For $x\in(0,1]$ we prove that the approximate invariant density converges pointwise to the true density at a rate $\frac{C^*}{x^{1+\alpha}}\cdot\frac{\ln m}{m}$, where $C^*$ is a computable fixed constant, $\alpha\in(0,1)$ is a fixed constant, and $m^{-1}$ is the mesh size of the discretization. To overcome the spectral difficulties and the unboundedness of the densities which we discussed above, we first induce the map and obtain a uniformly piecewise, expanding and onto map. Then we perform our discretization on the induced space. After that we pull back, both the invariant density and the approximate one to the full space and measure their difference in a weighted $L^{\infty}$-norm. Full details of our strategy is given in subsection \ref{strategy}.

\bigskip

In section \ref{expand}, we recall results on uniformly piecewise expanding and onto maps. Moreover, we introduce our discretization scheme and recall results about uniform approximations for invariant densities of uniformly piecewise expanding and onto maps. In section \ref{nonexp}, we introduce our non-uniformly expanding system, set up our strategy, and state our main results, Theorem \ref{main} and Corollary \ref{co1}. Section \ref{proofs} contains technical Lemmas and the proof of Theorem \ref{main}. Section \ref{Algo} presents an algorithm based on the result of Corollary \ref{co1} and discusses its feasibility.

\section{Preliminaries}\label{expand}
\subsection{A piecewise expanding system}\label{expanding}
Let $(\Delta,\mathfrak B,\hat\lambda)$ denote the measure space where $\Delta$ is a closed interval, $\mathfrak B$ is Borel $\sigma$-algebra and $\hat \lambda$ is normalized Lebesgue measure on $\Delta$. Let $\hat T:\Delta\to \Delta$ be a measurable transformation. We assume that there exists a countable partition $\mathcal P $ of $\Delta$, which consists of a sequence of intervals, $\mathcal P=\{I_i\}_{i=0}^\infty$, such that
\begin{enumerate}
\item for each $i=1,\dots, \infty$, $\hat T_i:=\hat T_{|\overset{\circ}{I}_i}$ is monotone, $C^{2}$ and it extends to a $C^2$ function on $\bar I_i$;
\item $\hat T_i(I_i)=\Delta$; i.e., for each $i=1,\dots, \infty$, $\hat T_i$ is onto;
\item there exists a constant $D>0$ such that  $\sup_i\sup_{x\in I_i}\frac{|\hat{T}''(x)|}{(\hat{T}'(x))^2}\le D$ ;
\item there exits a number $\gamma$ such that $\frac{1}{|\hat T_i'|}\le\gamma<1$.
\end{enumerate}
Let $\hat \ml:L^1\to L^1$ denote the transfer operator (Perron-Frobenius) \cite{Ba, BG} associated to $\hat T$:
$$\hat \ml f(x)=\sum_{y=\hat T^{-1}x}\frac{f(y)}{|\hat T'(y)|}.$$
Under the above assumptions, among other ergodic properties, it is well known (see for instance \cite{Bo}) $\hat T$ admits a unique invariant density $\hat f$; i.e. $\hat\ml\hat f=\hat f$. Moreover, $\hat\ml$ admits a spectral gap when acting on the space of Lipschitz continuous functions over $\Delta$ \cite{BB}\footnote{In \cite{BB}, a Lasota-Yorke inequality was obtained for Markov interval maps with a finite partition. The proof carries over for piecewise onto maps with a countable number of branches satisfying assumptions of subsection \ref{expanding}.}. We will denote by $BV(\Delta)$ the space of functions of bounded variation defined on the interval $\Delta$. Set $||\cdot||_{BV(\Delta)}:= V_{\Delta}+||\cdot||_{1,\Delta}$, where $ V_{\Delta}$ denotes the one-dimensional variation over $\Delta$. Then it is well known that $(BV(\Delta),  ||\cdot||_{BV(\Delta)})$ is a Banach space and $\hat\ml$ satisfies the following inequality (see \cite{P} for instance): there exists a constant $C_{LY}>0$ such that for any $f\in BV(\Delta)$, we have
\begin{equation}\label{LYBV}
V_{\Delta}\hat\ml f\le\gamma V_{\Delta} f+C_{LY}||f||_{1,\Delta}.
\end{equation}
Inequality \eqref{LYBV} is called the Lasota-Yorke inequality.
\subsection{Markov Discretization}\label{sec:markov}
We now introduce a discretization scheme which enables us to obtain rigorous uniform approximation of $\hat f$ the invariant density of $\hat T$. We use a piecewise linear approximations which was introduced by
Ding and Li \cite{DL}. Let $\eta=\{c_{i}\}_{i=0}^{m}$ be a partition of $\Delta$ into intervals. Since uniform partitions are the first choice for numerical work, we set $c_i - c_{i-1} = \frac{|\Delta|}{m}$, where $|\Delta|$ is the length of $\Delta$.  Everything we do can be easily modified for non-uniform partitions with only minor notational changes.  Let
$$\varphi_{i}=\chi_{[c_{i-1},c_{i}]} \text{ and } \phi_{i}(x)=m\int_0^{x}\varphi_{i}d\lambda.$$
Let $\psi_i$ denote a set of {\bf hat} functions over $\eta$:
\begin{equation}\label{hat}
\psi_{0}:=(1-\phi_{1})\,, \psi_{m}:=\phi_{m}\text{ and for }i=1,\dots,m-1\,, \psi_{i}:=(\phi_{i}-\phi_{i+1}).
\end{equation}
For $f \in L^1$, we set $I_i := [c_{i-1}, c_i]$ and
$$f_i := \frac{m}{|\Delta|} \int_{I_i} f \, dx, ~~i= 1, 2, \dots m,$$
the average of $f$ over the associated partition cell. For $f \in L^1$ we set
$$Q_m f := f_1\psi_0 + \sum_{i=1}^{m-1}  \frac{f_i + f_{i+1}}{2}  \psi_i  + f_m \psi_m$$
Obviously, the operator $Q_m$ retains good stochastic properties; i.e.,
\begin{itemize}
\item for $f\ge 0$, $Q_mf\ge 0$;
\item $\int Q_mf=\int f$.
\end{itemize}
 We now define a piecewise linear Markov discretization of $\hat\ml$ by
\begin{equation}\label{discrete}
\mathbb{P}_m:=Q_m\circ \hat\ml.
\end{equation}
Notice that $\mathbb{P}_m$ is a finite-rank Markov operator whose range is
contained in the space of continuous, piecewise linear functions with respect to $\eta$.
The matrix representation of $\mathbb{P}_m$ restricted to this finite-dimensional space and
with respect to the basis $\{\psi_i\}$ is a (row) stochastic matrix, with entries
$$p_{ij}:=m \int_{I_j} \hat\ml\psi_{i}\geq 0.$$
By the Perron-Frobenius Theorem for stochastic matrices \cite{LT}, $\mathbb P_m$ has a left invariant density $\hat f_m$; i.e.,
$$\hat f_m=\hat f_m \mathbb P_m.$$
The following theorem was proved in \cite{BB}:
\begin{theorem}\label{BB}
There exits a computable constant $\hat C$ such that for any $m\in\mathbb N$
$$||\hat f-\hat f_m||_{\infty}\le \hat C\frac{\ln m}{m}.$$
\end{theorem}
\begin{remark}
We recall that in \cite{BB} it was shown that the constant $\hat C$, which is independent of $m$, can be computed explicitly.
\end{remark}
\section{Pointwise Approximations for invariant densities of Maps with a neutral fixed point}\label{nonexp}
\subsection{The non-uniformly expanding system}\label{non} Let $I=[0,1]$ be the unit interval, $\lambda$ be Lebesgue measure on $[0,1]$. Let $T: I\rightarrow I$ be a piecewise smooth map with two branches. We assume that
\begin{itemize}
\item $T(0)=0$ and there is a $x_{0}\in(0,1)$ such that $T_{1}=T\mid_{[0,x_{0})}, T_{2}=T\mid_{[x_{0},1]}$ and
$T_{1}: [0,x_{0})\overset{\text{onto}}\rightarrow [0,1), T_{2}: [x_{0},1]\overset{\text{onto}}\rightarrow [0,1];$\\
\item $T_1$ is $C^1$ on $[0,x_0]$, $T_1$ is $C^2$ on $(0,x_0]$ and $T_2$ is $C^2$ on $[x_0,1]$.\\
\item $T'(0)=1$ and $T'(x)>1$ for  $x\in(0,x_{0})$ ; $|T'(x)|\geq \beta>1$ for  $x\in(x_{0},1);$\\
\item $T_{1}$  and $T'_{1}$ have the form
$$T_{1}(x)= x+x^{1+\alpha}+x^{1+\alpha}\delta_{0}(x),$$
$$T'_{1}(x)= 1+(1+\alpha)x^{\alpha}+x^{\alpha}\delta_{1}(x),$$
where, $0<\alpha<1$ and $\delta_{i}(x)\rightarrow 0$ as $x\rightarrow 0$ for $i = 0,1 $ with $\delta'_{0}(x)\geq 0.$
\end{itemize}
It is well known that $T$ admits a unique invariant density $f^*$ \cite{H, LSV, P, Y} and the system $(I,T,f^*\cdot\lambda)$ exhibits a polynomial mixing rate \cite{H, LSV, Y}. Moreover, it is well known \cite{H, LSV, Y} that the $T$-invariant density, $f^*$, is not an $L^{\infty}$-function. In particular, near $x=0$, $f^*(x)$ behaves like $x^{-\alpha}$. Despite this difficulty, we will show that, for any $x\in(0,1]$, one can obtain rigorous pointwise approximation of $f^*(x)$.
\subsection{Strategy and the statement of the main result}\label{strategy}
Recall that $\alpha\in(0,1)$. We first define a suitable Banach space that contains $f^*$. More precisely, let $\mathcal{B}$ denote the set of continuous functions on $(0,1]$ with the norm $$\parallel f\parallel_{\mathcal{B}}=\sup\limits_{x\in(0,1]}|x^{1+\alpha}f(x)|.$$
When equipped with the norm $\parallel \cdot\parallel_{\mathcal{B}}$, $\mathcal{B}$ is a Banach space\footnote{In what follows, we only use the metric properties of $\mathcal B$. In particular, the completeness of $\mathcal B$ is not needed in our proofs.}. The fact that $f^*\in\mathcal B$ follows from Lemma 3.3 of \cite{H}. Our strategy for obtaining pointwise approximation $f^*$ consists of the following steps:
\begin{enumerate}
\item We first induce $T$ on $\Delta\subset I$ and obtain a $\hat T$ which satisfies the assumptions of subsection \ref{expanding}.
\item On $\Delta$, we use Theorem \ref{BB} to say that $\hat f_m$, the invariant density of the discretized operator $\mathbb P_m:=Q_m\circ\hat\ml$, defined in equation \eqref{discrete}, provides a uniform approximation of $\hat f$ the $\hat T$-invariant density.
\item Next we write $f^*$ in terms of $\hat f$, and define a function $f_m$ as the `pullback' of $\hat f_m$.
\item Finally, we use steps (2) and (3) to prove that $||f^*- f_m||_{\mathcal B}\le C^*\frac{\ln m}{m}$, and deduce a pointwise approximation of $f^*$.
\end{enumerate}
\subsubsection{The induced system}
We induce $T$ on $\Delta:=[x_0,1]$. For $n\ge 0$ we define
$$x_{n+1} = T_{1}^{-1}(x_{n}).$$
Set
$$W_{0}:=(x_0,1),\text{and } W_{n}:=(x_{n},x_{n-1}),\, n\ge 1.$$
For $n\ge 1$, we define
$$Z_n:=T_2^{-1}(W_{n-1}).$$
Then we define the induced map $\hat T:\Delta\to\Delta$ by
\begin{equation}\label{induced}
\hat T(x)=T^{n}(x)\text{ for $x\in Z_{n}$.}
\end{equation}
Observe that
$$T(Z_{n})= W_{n-1}\text{ and } \tau_{Z_{n}}=n,$$
where $\tau_{Z_{n}}$ is the first return time of $Z_{n}$ to $\Delta$. An example of the map $T$ and its induced counterpart $\hat T$ are shown in Figures \ref{Fig1} and \ref{Fig2} respectively.
It is well known (see for instance \cite{Y}) that the $\hat T$ defined in \eqref{induced} satisfies the assumptions of subsection \ref{expanding}, and, by Theorem \ref{BB}, one can obtain a rigorous \textit{uniform} approximation of its invariant density $\hat f$. Moreover, by Lemma 3.3 of \cite{BV}, $f^*$, the invariant density of $T$, can be written in terms of $\hat f$:
\begin{figure}[htbp] 
   \centering
   \includegraphics[width=2.5in]{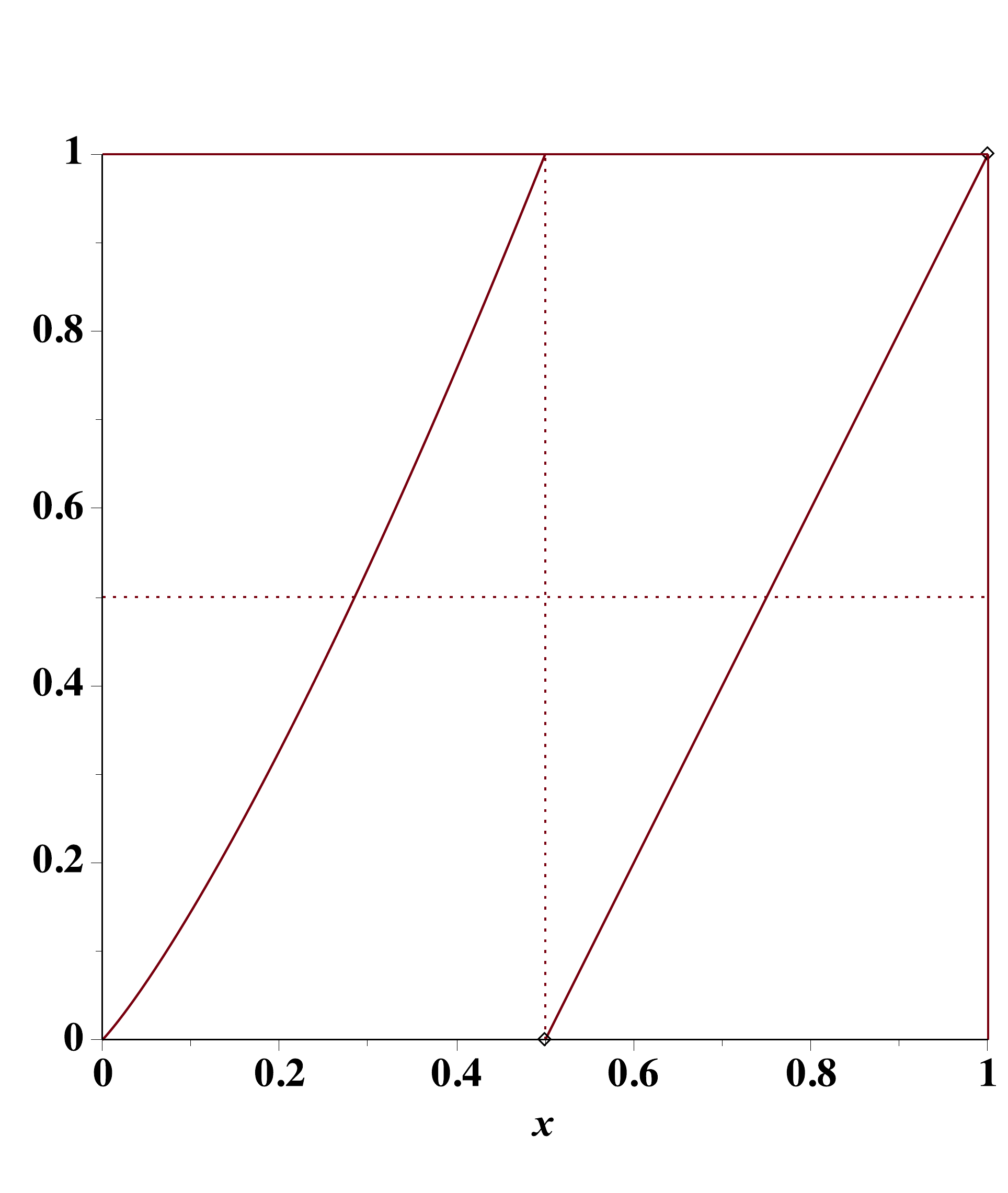}
   \caption{A typical example of a map $T$ which belongs to the family defined in subsection \ref{non}.}   \label{Fig1}
\end{figure}
\begin{figure}[htbp] 
   \centering
   \includegraphics[width=4in]{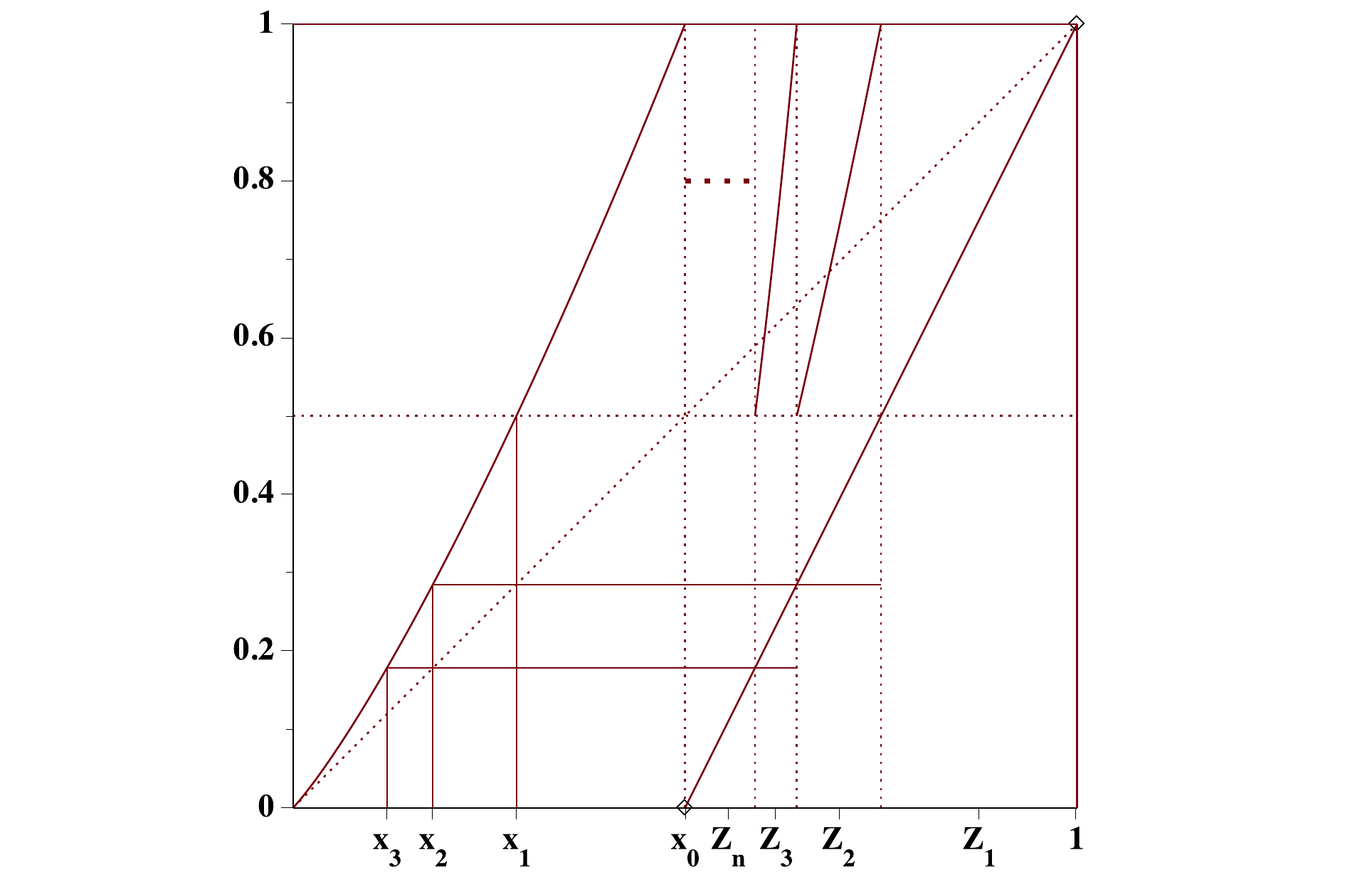}
   \caption{This figure shows the induced map $\hat T$ corresponding to the map $T$ of Figure \ref{Fig1}.}   \label{Fig2}
\end{figure}

\begin{equation}\label{density}
f^*(x)=\left\{\begin{array}{ccc}
c_{\tau} \hat f(x)&\mbox{for $x\in \Delta$}\\
\text{ }\\
c_{\tau}\sum\limits_{n=1}\limits^{\infty}\left(\frac{\hat f(T_{2}^{-1}T_{1}^{-(n-1)}x)}{|DT^{(n)}(T_{2}^{-1}T_{1}^{-(n-1)}x)|} \right) &\mbox{for $x\in I\setminus\Delta$}
\end{array}
\right. ,
\end{equation}
where $\hat f$ is the $\hat T$-invariant density, $c^{-1}_{\tau}=\sum_{k=1}^{\infty} \tau_{Z_{k}} \hat\mu(Z_{k})$, and $\hat\mu=\hat f\cdot\hat\lambda$.
\subsubsection{The approximate density and the statement of the main result}
Set
\begin{equation}\label{approx}
f_m(x)\overset{\text{def}}{:=}\left\{\begin{array}{ccc}
c_{\tau,m}\hat f_{m}(x)&\mbox{for $ x\in \Delta$}\\
c_{\tau,m}\sum\limits_{n=1}\limits^{\infty} \left(\frac{\hat f_{m}(T_{2}^{-1}T_{1}^{-(n-1)}x)}{|DT^{(n)}(T_{2}^{-1}T_{1}^{-(n-1)}x)|}\right) &\mbox{for $x\in I\setminus\Delta$}
\end{array}
\right. ,
\end{equation}
where $\hat f_m=\mathbb P_m\hat f_m$, and $\mathbb P_m$ is the Markov discretization of $\hat\ml$ defined in \eqref{discrete}, $c^{-1}_{\tau,m}=\sum_{k=1}^{\infty} \tau_{Z_{k}} \hat\mu_m(Z_{k})$, and $\hat\mu_m=\hat f_m\cdot\hat\lambda$. The next result shows that the function $f_m$ defined in \eqref{approx} provides a rigorous pointwise approximation of $f^*$.
\begin{theorem}\label{main}
For any $m\in\mathbb N$ we have
$$|| f^*- f_m||_{\mathcal B}\le C^*\frac{\ln m}{m},$$
where
$$C^*=\hat C\left(1+\frac{x^{1+\alpha}_{0}}{\beta}+M(1+\alpha)\right)C_4;$$
in particular, $\hat C$ is the computable constant of Theorem \ref{BB},
$$M:= \frac{C_{1}^{1+\alpha}e^{2C_{0}C_{1}^{2\alpha}}}{\beta},$$
$$C_0:=\frac{\alpha(1+\alpha)}{2}[1+2\delta_{0}(x_{0})+\delta^{2}_{0}(x_{0})],\,\ C_1:=(2[2^{\frac{1}{\alpha}}-1])^{1/\alpha},$$
$$C_4:=1+C_3(\frac{C_{LY}}{1-\gamma}+\frac{1}{|\Delta|}), C_3:=\frac{1}{\beta}+\frac{C_{2}}{\beta|\Delta|}(\alpha+\frac{2-\alpha}{1-\alpha}),$$
 and
 $$C_{2} = \frac{1-x_0}{x^{1+\alpha}_0}  2^{1+\frac{1}{\alpha}}[2^{\frac{1}{\alpha}}-1]^{1+\frac{1}{\alpha}}.$$
\end{theorem}
As a direct consequence of the Theorem \ref{main} we obtain the required pointwise approximation of $f^*$:
\begin{corollary}\label{co1}
For any $x\in(0,1]$ we have
$$| f^*(x)- f_m(x)|\le \frac{C^*}{x^{1+\alpha}}\frac{\ln m}{m}.$$
\end{corollary}
\begin{proof}
For $x\in(0,1]$, we have
$$| f^*(x)- f_m(x)|=\frac{1}{x^{1+\alpha}}|x^{1+\alpha}( f^*(x)- f_m(x))|\le \frac{1}{x^{1+\alpha}}|| f^*- f_m||_{\mathcal B}\le\frac{1}{x^{1+\alpha}}C^*\frac{\ln m}{m}.$$
\end{proof}
\section{Proofs}\label{proofs}
\subsection{Technical lemmas}
We first introduce notation of certain functions which appear in the proof of Theorem \ref{main}. For $x\in I\setminus \Delta$ set:
\begin{eqnarray*}
&&g(x):=\frac{(\frac{T_{1}x}{x})^{1+\alpha}}{T'_{1}(x)},\\
&&G_{1}(x):=\frac{x^{1+\alpha}}{|T_{2}'(T_{2}^{-1}x)|};\quad\text{ and for }n\geq 2,\quad G_{n}(x):=\frac{x^{1+\alpha}}{|DT^{(n)}(T_{2}^{-1}T_{1}^{-(n-1)}x)|}.
\end{eqnarray*}
\begin{lemma}\label{le2.1} For $x\in I\setminus\Delta$, we have
$$[1+x^{\alpha}+x^{\alpha}\delta_{0}(x)]^{1+\alpha}\leq 1+(1+\alpha)[x^{\alpha}+x^{\alpha}\delta_{0}(x)]+\frac{\alpha(1+\alpha)}{2}[x^{\alpha}+x^{\alpha}\delta_{0}(x)]^{2}.$$
\end{lemma}
\begin{proof}
Let $$\phi_1(x):=[1+x^{\alpha}+x^{\alpha}\delta_{0}(x)]^{1+\alpha}$$
and
$$\phi_2(x):=1+(1+\alpha)[x^{\alpha}+x^{\alpha}\delta_{0}(x)]+\frac{\alpha(1+\alpha)}{2}[x^{\alpha}+x^{\alpha}\delta_{0}(x)]^{2}.$$
Note that $\phi_1(0)=\phi_2(0)=1$. Therefore, to prove the lemma, it is enough to prove that $\phi_1'(x)\le\phi'_2(x).$ We have:
\begin{eqnarray*}
&&\phi_1'(x)=(1+\alpha)(1+\xi(x))^{\alpha}\xi'(x)\\
&&\phi'_2(x)=(1+\alpha)(1+\alpha\xi(x))\xi'(x),
\end{eqnarray*}
where $\xi(x):=x^{\alpha}+x^{\alpha}\delta_{0}(x)\geq 0$\footnote{It is obvious that $\xi(0)=0$ and for $x>0$, $\xi(x)>0.$}. Notice that $\xi'(x)\geq 0$. Thus, we only need to show that
\begin{equation}\label{est}
(1+\xi(x))^{\alpha}\leq(1+\alpha\xi(x)).
\end{equation}
Indeed, \eqref{est} holds because
$(1+\xi(0))^{\alpha}=(1+\alpha\xi(0))=1$ and
$$[(1+\xi(x))^{\alpha}]'=\frac{\alpha}{(1+\xi(x))^{1-\alpha}}\xi'(x)\leq\alpha\xi'(x)=[1+\alpha\xi(x)]'.$$

\end{proof}

\begin{lemma}\label{le2.2} For $x\in I\setminus\Delta$, we have
$g(x)\leq 1+ C_{0}x^{2\alpha},$ where
$$C_{0}= \frac{\alpha(1+\alpha)}{2}[1+2\delta_{0}(x_{0})+\delta^{2}_{0}(x_{0})].$$
\end{lemma}
\begin{proof}
Using Lemma \ref{le2.1}, we have:
\begin{eqnarray*}
g(x)&=&\frac{(\frac{T_{1}x}{x})^{1+\alpha}}{T'_{1}(x)}=\frac{[1+x^{\alpha}+x^{\alpha}\delta_{0}(x)]^{1+\alpha}}{1+(1+\alpha)x^{\alpha}+x^{\alpha}\delta_{1}(x)}\\
&\leq&\frac{1+(1+\alpha)[x^{\alpha}+x^{\alpha}\delta_{0}(x)]+\frac{\alpha(1+\alpha)}{2}[x^{\alpha}+x^{\alpha}\delta_{0}(x)]^{2}}
{1+(1+\alpha)x^{\alpha}+x^{\alpha}\delta_{1}(x)}\\
&=&\frac{1+(1+\alpha)[x^{\alpha}+x^{\alpha}\delta_{0}(x)]}{1+(1+\alpha)x^{\alpha}+x^{\alpha}\delta_{1}(x)}+
\frac{\frac{\alpha(1+\alpha)}{2}[x^{\alpha}+x^{\alpha}\delta_{0}(x)]^{2}}{1+(1+\alpha)x^{\alpha}+x^{\alpha}\delta_{1}(x)}\\
&\leq&1 + \frac{\alpha(1+\alpha)}{2}[x^{\alpha}+x^{\alpha}\delta_{0}(x)]^{2}\\
&=&1+\frac{\alpha(1+\alpha)}{2}(1+2\delta_{0}(x)+\delta^{2}_{0}(x))x^{2\alpha}\leq1+C_{0}x^{2\alpha}.\\
\end{eqnarray*}
\end{proof}

\begin{lemma}\label{le2.3} Let $x_{n}=T_{1}^{-n}x_{0}.$ For $ n\geq 1, x_{n}\leq C_{1}n^{-\frac{1}{\alpha}},$ where $C_{1} = (2[2^{\frac{1}{\alpha}}-1])^{1/\alpha}.$
\end{lemma}
\begin{proof}
Observe that $C_{1}>1\geq T^{-1}_{1}(x_{0})=x_{1}.$ Therefore, the lemma is true for $n=1$.
Next, for $n\geq2,$ we suppose that $x_{n-1}\leq C_{1}(n-1)^{-\frac{1}{\alpha}},$ and prove that $x_{n}\leq C_{1}n^{-\frac{1}{\alpha}}.$
If it is false, that is $x_{n}> C_{1}n^{-\frac{1}{\alpha}},$ then by our inductive statement on $x_{n-1}$, we have:
$$C_{1}(n-1)^{-\frac{1}{\alpha}}\geq x_{n-1}= T_{1}(x_{n})>C_{1}n^{-\frac{1}{\alpha}}[1+C^{\alpha}_{1}n^{-1}+C^{\alpha}_{1}n^{-1}\delta_{0}(C_{1}n^{-\frac{1}{\alpha}})].$$
This is equivalent to
$$n[(1+\frac{1}{n-1})^{\frac{1}{\alpha}}-1]> C^{\alpha}_{1}[1+\delta_{0}(C_{1}n^{-\frac{1}{\alpha}})].$$
By convexity of  the function $z^{\frac{1}{\alpha}},$ it follows $\frac{n}{n-1}[2^{\frac{1}{\alpha}}-1]> C^{\alpha}_{1}[1+\delta_{0}(C_{1}n^{-\frac{1}{\alpha}})],$ that is
$$C^{\alpha}_{1}< \frac{n}{n-1}[2^{\frac{1}{\alpha}}-1]/[1+\delta_{0}(C_{1}n^{-\frac{1}{\alpha}})]<2[2^{\frac{1}{\alpha}}-1]=C^{\alpha}_{1}.$$
A contradiction. Therefore, $x_{n}\leq C_{1}n^{-\frac{1}{\alpha}}$, and this completes the proof of the lemma.
\end{proof}
\begin{lemma}\label{Gn}
For $x\in I\setminus\Delta$, we have
$$G_{1}(x)\leq\frac{x_{0}^{1+\alpha}}{\beta}.$$
and for $n\ge 2$,
$$G_n(x)\le M(n-1)^{-(1+\frac{1}{\alpha})},$$
where $M = \frac{C_{1}^{1+\alpha}e^{2C_{0}C_{1}^{2\alpha}}}{\beta}.$
\end{lemma}
\begin{proof}
For $n=1$, it is easy to see that
$$G_{1}(x)\leq\frac{x_{0}^{1+\alpha}}{\beta}.$$
For $n\geq 2$, we have
\begin{equation}\label{est1}
\begin{split}
G_{n}(x)&=\frac{x^{1+\alpha}}{|DT^{(n)}(T_{2}^{-1}T_{1}^{-(n-1)}x)|}\\
&=\frac{x^{1+\alpha}}{|D(T_1\circ T_1\circ\cdots\circ T_1\circ T_2)(T_{2}^{-1}T_{1}^{-(n-1)}x)|}\\
&=\frac{x^{1+\alpha}}{T'_{1}(T_{1}^{-1}x)\cdot T'_{1}(T_{1}^{-2}x)\cdots T'_{1}(T_{1}^{-(n-1)}x)\cdot |T'_{2}(T_{2}^{-1}T_{1}^{-(n-1)}x)|}\\
&=\frac{(\frac{x}{T_{1}^{-1}x})^{1+\alpha}}{T'_{1}(T_{1}^{-1}x)}\cdot\frac{(\frac{T_{1}^{-1}x}{T_{1}^{-2}x})^{1+\alpha}}{T'_{1}(T_{1}^{-2}x)}
\cdots\frac{(\frac{T_{1}^{-(n-2)}x}{T_{1}^{-(n-1)}x})^{1+\alpha}}{T'_{1}(T_{1}^{-(n-1)}x)}\cdot\frac{(T_{1}^{-(n-1)}x)^{1+\alpha}}{|T'_{2}(T_{2}^{-1}T_{1}^{-(n-1)}x)|}\\
&=g(T_{1}^{-1}x)\cdot g(T_{1}^{-2}x)\cdots g(T_{1}^{-(n-1)}x)\cdot\frac{(T_{1}^{-(n-1)}x)^{1+\alpha}}{|T'_{2}(T_{2}^{-1}T_{1}^{-(n-1)}x)|}\\
&\leq g(T_{1}^{-1}x)\cdot g(T_{1}^{-2}x)\cdots g(T_{1}^{-(n-1)}x)\cdot\frac{(T_{1}^{-(n-1)}x)^{1+\alpha}}{\beta}.
\end{split}
\end{equation}
By Lemmas \ref{le2.2} and \ref{le2.3}, for any $k\geq 1, x\in[0,x_{0})$, we have
\begin{equation}\label{est2}
\begin{split}
g(T^{-k}_{1}(x))&\leq1+ C_{0}(T^{-k}_{1}(x))^{2\alpha}\leq1+ C_{0}(T^{-k}_{1}(x_{0}))^{2\alpha}\\
&=1+ C_{0}(x_{k})^{2\alpha}\le1+C_{0}C_{1}^{2\alpha}k^{-2}.
\end{split}
\end{equation}
Therefore, using \eqref{est1} and \eqref{est2}, for $n\geq 2$, we obtain:
\begin{eqnarray*}
G_{n}(x)&=&\prod\limits_{k=1}\limits^{n-1}g(T^{-k}_{1}(x))\cdot\frac{(T^{-(n-1)}_{1}(x))^{1+\alpha}}{\beta}\\
&\leq&\prod\limits_{k=1}\limits^{n-1}(1+C_{0}C_{1}^{2\alpha}k^{-2})\cdot\frac{C_{1}^{1+\alpha}(n-1)^{-(1+\frac{1}{\alpha})}}{\beta}\\
&=&\exp\{\sum\limits_{k=1}\limits^{n-1}\ln(1+C_{0}C_{1}^{2\alpha}k^{-2})\}\cdot\frac{C_{1}^{1+\alpha}(n-1)^{-(1+\frac{1}{\alpha})}}{\beta}\\
&\leq&\exp\{\sum\limits_{k=1}\limits^{n-1}C_{0}C_{1}^{2\alpha}k^{-2}\}\cdot\frac{C_{1}^{1+\alpha}(n-1)^{-(1+\frac{1}{\alpha})}}{\beta}\\
&\leq&\exp\{C_{0}C_{1}^{2\alpha}(2-\frac{1}{n-1})\}\cdot\frac{C_{1}^{1+\alpha}(n-1)^{-(1+\frac{1}{\alpha})}}{\beta}\\
&\leq& M(n-1)^{-(1+\frac{1}{\alpha})}.
\end{eqnarray*}
\end{proof}
\begin{lemma}\label{new1}
$$\sum\limits_{n=1}\limits^{\infty} n\cdot \hat\lambda(Z_{n})\leq C_3,$$ where $C_3=\frac{1}{\beta}+\frac{C_{2}}{\beta(1-x_0)}(\alpha+\frac{2-\alpha}{1-\alpha})$ and
$C_{2} = \frac{1-x_0}{x^{1+\alpha}_0}  2^{1+\frac{1}{\alpha}}[2^{\frac{1}{\alpha}}-1]^{1+\frac{1}{\alpha}}.$
\end{lemma}
\begin{proof}
By Lemma \ref{le2.3}, we have $ \lambda(W_n)=x_{n-1}-x_n= T_1(x_n)-x_n=\frac{1-x_0}{x^{1+\alpha}_0} x_n^{1+\alpha}\leq \frac{1-x_0}{x^{1+\alpha}_0}C^{1+\alpha}_1n^{-(1+\frac{1}{\alpha})}
= C_{2}n^{-(1+\frac{1}{\alpha})}.$
Since $T_2(Z_n)=W_{n-1}$, we have
\begin{eqnarray*}
\sum\limits_{n=1}\limits^{\infty} n\cdot \lambda(Z_{n})
&\leq& \lambda(Z_{1})+\sum\limits_{n=2}\limits^{\infty} n\cdot\frac{ \lambda(W_{n-1})}{\beta}\\
&\leq& \frac{1-x_0}{\beta}+\sum\limits_{n=2}\limits^{\infty}\frac{n(x_{n-2}-x_{n-1})}{\beta}\\
&=&\frac{1-x_0}{\beta}+\sum\limits_{n=1}\limits^{\infty}\frac{(n+1)(x_{n-1}-x_{n})}{\beta}\\
&=&\frac{1-x_0}{\beta}+\sum\limits_{n=1}\limits^{\infty}\frac{n(x_{n-1}-x_{n})}{\beta}+\sum\limits_{n=1}\limits^{\infty}\frac{(x_{n-1}-x_{n})}{\beta}\\
&\leq&\frac{1-x_0}{\beta}+\sum\limits_{n=1}\limits^{\infty} \frac{C_{2}}{\beta}n^{-\frac{1}{\alpha}}+
\sum\limits_{n=1}\limits^{\infty} \frac{C_{2}}{\beta}n^{-(1+\frac{1}{\alpha})}\\
&\leq&\frac{1-x_0}{\beta}+\frac{C_{2}}{\beta}(1+\int\limits_{1}\limits^{\infty}x^{-\frac{1}{\alpha}}dx)+
\frac{C_{2}}{\beta}(1+\int\limits_{1}\limits^{\infty}x^{-(1+\frac{1}{\alpha})}dx)\\
&=&\frac{1-x_0}{\beta}+\frac{C_{2}}{\beta}(\alpha+\frac{2-\alpha}{1-\alpha})=(1-x_0)\cdot C_3.
\end{eqnarray*}
This completes the proof of the lemma since $\hat{\lambda}(\cdot)= \frac{\lambda(\cdot)}{1-x_0}$.
\end{proof}
\begin{lemma}\label{new2}
We have
$$|c_{\tau,m}-c_{\tau}|\leq C_3\cdot \hat C\frac{\ln m}{m}.$$
\end{lemma}
\begin{proof}
Using the fact that $c_{\tau}\le 1$,  $c_{\tau,m}\le 1$ and Theorem \ref{BB}, we have
\begin{eqnarray*}
|c_{\tau,m}-c_{\tau}|&\leq&\left|\frac{1}{\sum\limits_{k=1}\limits^{\infty} \tau_{Z_{k}} \hat\mu_{m}(Z_{k})}-\frac{1}{\sum\limits_{k=1}\limits^{\infty} \tau_{Z_{k}} \hat\mu(Z_{k})}\right|\\
&=&\left|\frac{\sum\limits_{k=1}\limits^{\infty}k[\hat\mu(Z_{k})-\hat\mu_m(Z_{k})]}{\sum\limits_{k=1}\limits^{\infty} \tau_{Z_{k}} \hat\mu_{m}(Z_{k})\cdot\sum\limits_{k=1}\limits^{\infty} \tau_{Z_{k}} \hat\mu(Z_{k})}
\right|\\
&\leq&\left(\sum\limits_{k=1}\limits^{\infty}k\int\limits_{Z_k}|\hat f-\hat f_m|d\hat\lambda\right)\\
&\leq&||\hat f-\hat f_m||_{\infty}\left(\sum\limits_{k=1}\limits^{\infty}k\hat\lambda(Z_k)\right)\\
&\leq&\hat C\frac{\ln m}{m}\cdot C_3.
\end{eqnarray*}
In the last estimate, we have used Lemma \ref{new1}.
\end{proof}
We now have all our tools ready to prove Theorem \ref{main}.
\begin{proof} ({\bf of Theorem \ref{main}})
Using \eqref{density} and \eqref{approx}, we have
\begin{equation}\label{eq1}
\begin{split}
&||f^*-f_m||_{\mathcal B}=\sup_{x\in(0,1]}|x^{1+\alpha}(f^*(x)-f_m(x))|\\
&\le \sup_{x\in I\setminus\Delta}|x^{1+\alpha}(f^*(x)-f_m(x))|+ \sup_{x\in \Delta}|x^{1+\alpha}(f^*(x)-f_m(x))|\\
&= \sup_{x\in I\setminus\Delta}|\sum\limits_{n=1}\limits^{\infty}\frac{x^{1+\alpha}}{DT^{(n)}(T_{2}^{-1}T_{1}^{-(n-1)}x)}(c_{\tau}\hat f(T_{2}^{-1}T_{1}^{-(n-1)}x)-c_{\tau,m}\hat f_m(T_{2}^{-1}T_{1}^{-(n-1)}x))|\\
&+ \sup_{x\in \Delta}|x^{1+\alpha}(c_{\tau}\hat f(x)-c_{\tau,m}\hat f_m(x))|.
\end{split}
\end{equation}
Notice that for $x\in I\setminus\Delta$, and $n\ge 1$, $z_n:=T_{2}^{-1}T_{1}^{-(n-1)}x\in\Delta$. Then using the fact that $c_{\tau}\le 1$,  $c_{\tau,m}\le 1$, Theorem \ref{BB}, Lemma \ref{new2}, and \eqref{eq1}, we obtain:
\begin{equation}\label{eq2}
\begin{split}
&||f^*-f_m||_{\mathcal B}\le \sup_{x\in I\setminus\Delta}|\sum\limits_{n=1}\limits^{\infty}\frac{x^{1+\alpha}}{DT^{(n)}(T_{2}^{-1}T_{1}^{-(n-1)}x)}|\cdot \sup_{z_n\in \Delta}|(c_{\tau}\hat f(z_n)-c_{\tau,m}\hat f_m(z_n)|\\
&+ \sup_{x\in \Delta}|c_{\tau}\hat f(x)-c_{\tau,m}\hat f_m(x)|\\
&\le\sup_{x\in I\setminus\Delta}|\sum\limits_{n=1}\limits^{\infty}\frac{x^{1+\alpha}}{DT^{(n)}(T_{2}^{-1}T_{1}^{-(n-1)}x)}|\times\\
&\hskip 5cm\left( \sup_{z_n\in \Delta}|\hat f(z_n)-\hat f_m(z_n)|+|c_{\tau}-c_{\tau,m}|\sup_{z_n\in \Delta}|\hat f(z_n)|\right)\\
&+ \sup_{x\in \Delta}|\hat f(x)-\hat f_m(x)|+|c_{\tau}-c_{\tau,m}|\sup_{x\in \Delta}|\hat f(x)|\\
&\le\hat C\frac{\ln m}{m}\left(\sup_{x\in I\setminus\Delta}\sum_{n=1}^\infty |G_n(x)|(1+C_3\sup_{z_n\in \Delta}|\hat f(z_n)|)+(1+C_3\sup_{x\in \Delta}|\hat f(x)|)\right).
\end{split}
\end{equation}
Since $\hat f\in BV(\Delta)$, we have $\sup_{x\in\Delta}|\hat f(x)|\le V_{\Delta}\hat f+\frac{1}{1-x_0}||\hat f||_{1,\Delta}$. Therefore, using the Lasota-Yorke inequality \eqref{LYBV}, we obtain $\sup_{x\in\Delta}|\hat f(x)|\le\frac{C_{LY}}{1-\gamma}+\frac{1}{1-x_0}$.
Using Lemma \ref{Gn} and \eqref{eq2}, we obtain:
\begin{eqnarray*}
||f^*-f_m||_{\mathcal B}&\le&C_4\hat C\frac{\ln m}{m}\left(1+\frac{x^{1+\alpha}_{0}}{\beta}+\sum\limits_{n=2}\limits^{\infty} M(n-1)^{-(1+\frac{1}{\alpha})}\right)\\
&=&C_4\hat C\frac{\ln m}{m}\left(1+\frac{x^{1+\alpha}_{0}}{\beta}+M\sum\limits_{n=1}\limits^{\infty} n^{-(1+\frac{1}{\alpha})}\right)\\
&\le&C_4\hat C\left(1+\frac{x^{1+\alpha}_{0}}{\beta}+M(1+\alpha)\right)\cdot\frac{\ln m}{m}.
\end{eqnarray*}
\end{proof}

\section{Algorithm and Feasibility}\label{Algo}
Given a map $T$ satisfying the conditions of section \ref{non}, and $x^\ast\in (0,1],$ we provide an algorithm based on corollary \ref{co1} that can be used to approximate $f^\ast(x^\ast),$ the $T-$invariant density at the point $x^\ast,$ up to a pre-specified  approximation error $R.$

\subsection{Algorithm and output}
\begin{enumerate}
\item Compute the constants $\beta,M,C_4$ which appear in Theorem \ref{main}.

\item Compute an upper bound on the constant $\hat C$ which appears in Theorem \ref{BB}.

\item Then use $(1) , (2)$ to compute $C^\ast$ which appear in Theorem \ref{main}.

\item Find $m^\ast,$ number of bins, such that
$$\frac{\ln m^\ast}{m^\ast}\leq\frac{{x^\ast}^{1+\alpha}}{C^\ast}\cdot\frac{R}{3}.$$

\item Compute an approximate fixed point\footnote{Recall that $\hat f_m$ is the fixed point of the finite rank operator $\mathbb P_m$ defined in \eqref{discrete}. Here ${\tilde{\hat{f}}_m}$ is the computer approximation of $\hat f_m$; i.e., (5) of the Algorithm takes care of the computer roundoff errors in computing the fixed point of $\mathbb P_m$. Since $\hat f_m>0$, we also ask in this computation that ${\tilde{\hat f}}_m>0$. Note that the strict positivity of $\hat f_m$ follows from the fact that the induced map $\hat T$ is a piecewise onto map, which implies that the matrix representation of $\mathbb P$ is irreducible, and consequently its Perron eigenvector is strictly positive (see Perron-Frobenius Theorem \cite{LT}). }. ${\tilde{\hat{f}}_m}>0$ of $\mathbb P_m$ so that $||{\tilde{\hat{f}}_m}-\hat f_m||_{\infty}\le\eps$ on $\Delta$, where $\eps$ is chosen such that 
$$|{\tilde{f}_m}(x^*)-f_m(x^*)|\le \frac{R}{3},$$
where $${\tilde{f}_m}(x)={\tilde c}_{\tau,m}\sum\limits_{n=1}\limits^{\infty} \left(\frac{{\tilde{\hat{f}}_m}(T_{2}^{-1}T_{1}^{-(n-1)}x)}{|DT^{(n)}(T_{2}^{-1}T_{1}^{-(n-1)}x)|}\right),$$
${\tilde c}^{-1}_{\tau,m}:=\sum\limits_{k=1}\limits^{\infty}k{\tilde{\hat\mu}}_m(Z_k)$ and ${\tilde{\hat\mu}}={\tilde{\hat{f}}_m}\cdot\hat\lambda$.

\item Find $N^\ast$ such that 
$$|{\tilde c}_{\tau,m}(N^*)\sum\limits_{n=1}^{N^\ast}\limits \left(\frac{{\tilde{\hat{f}}_m}(T_{2}^{-1}T_{1}^{-(n-1)}x^\ast)}{|DT^{(n)}(T_{2}^{-1}T_{1}^{-(n-1)}x^\ast)|}\right)-{\tilde f}_m(x^*)|\leq \frac{R}{3},$$
where ${\tilde c}^{-1}_{\tau,m}(N^\ast):=\sum\limits_{k=1}\limits^{N^\ast}k{\tilde{\hat\mu}}_m(Z_k)$.

\item The approximate value of $f^*(x^*)$ is given by\footnote{It is very important to notice that the approximation ${\tilde{f}}_{m, N^*}(x^*)$ is a {\em finite} sum.}:  
$${\tilde{f}}_{m, N^{*}}(x^*):={\tilde c}_{\tau,m}(N^*)\sum\limits_{n=1}^{N^\ast}\left(\frac{{\tilde{\hat{f}}_m}(T_{2}^{-1}T_{1}^{-(n-1)}x^\ast)}{|DT^{(n)}(T_{2}^{-1}T_{1}^{-(n-1)}x^\ast)|}\right).$$
\end{enumerate}
\subsection{Feasibility}
\begin{itemize}
\item For $(1),$ once the map $T$ is given, the constants $\beta,M,C_4$ can be computed analytically.

\item For $(2),$ the constant $\hat C$ appears in the approximation done on the induced system (See Theorem \ref{BB}). The induced system is a uniformly expanding map. The computation of $\hat C$ can be done following the ideas of \cite{BB} which is based on the spectral stability result of \cite{KL}.
    
\item $(3)$ is a consequence of $(1)$ and $(2).$

\item Once $C^\ast$ is computed in $(3),$ with $m:=m^\ast$ we know, from Corollary \ref{co1}, that $|f_m(x^*)-f^*(x^*)|\le R/3$.

\item For $(5)$ we should find out how small $\eps$ should be to ensure $$|{\tilde f}_m(x^*)-f_m(x^*)|\le\frac{R}{3}.$$ 
We propose the following method. To work out explicitly all the constants needed in verifying (5), we suppose $x^\ast\in W_k$, and $T_1(x)=x+2^{\alpha}x^{1+\alpha}$. In the following estimates \eqref{thesum}, \eqref{newc} and \eqref{hatfbd}, we prepare the ingredients to achieve our job. Firstly, following exactly the argument of Lemma 5.2 in \cite{BV}, for $x^*\in W_k$,
\begin{equation}\label{ineq1}
|DT^{(n)}(T_2^{-1}T_1^{-(n-1)}x^\ast)|\geq \beta(\frac{n+k}{k+2})^{\eta_k},
\end{equation}
where 
$$\eta_k=\frac{d(k+2)}{k+2+d},\hskip 0.5cm d=(1+\alpha)2^{\alpha}[\frac{1}{2(1+\alpha)^{\frac{1}{\alpha}}}+d_1]^{\alpha}>1,$$ and 
$$d_1=\frac{1}{2(\alpha)^{\frac{1}{\alpha}}}-\frac{1}{2(1+\alpha)^{\frac{1}{\alpha}}}.$$
Consequently, $\eta_k>1,$ and for $x^*\in W_k$,
\begin{equation}\label{thesum}
\sum_{n=1}^{\infty}\frac{1}{|DT^{(n)}(T_2^{-1}T_1^{-(n-1)}x^\ast)|}\le\frac{1}{\beta}(\frac{k+2}{k})^{\eta_{k}}\frac{k}{\eta_k-1}.
\end{equation}
Secondly, using the same argument as the one in the proof of Lemma \ref{new2}, we have
\begin{equation}\label{newc}
|{\tilde c}_{\tau,m}-c_{\tau,m}|\le C_3\cdot ||{\tilde{\hat f}}_m-\hat f_m||_{\infty}.
\end{equation}
Thirdly, it is well known\footnote{See for instance \cite{DL} Lemma 2.3.} that for $g\in BV(\Delta)$ we have $V_{\Delta}Q_m g\le g$, where $Q_m$ is the discretization defined in subsection \ref{sec:markov}. Therefore, $\mathbb P_m$ satisfies the same Lasota-Yorke inequality \eqref{LYBV} as $\hat{\mathcal L}$. In particular, this implies that $\hat f_m$, the $\mathbb P_m$ fixed point, satisfies:
\begin{equation}\label{hatfbd}
||\hat f_m||_{\infty}\le\frac{C_{LY}}{1-\gamma}+\frac{1}{|\Delta|}.
\end{equation}
Consequently, using  \eqref{thesum}, \eqref{newc}, \eqref{hatfbd} and that ${\tilde{c}}_{\tau,m}\le 1$, we obtain
\begin{equation}\label{epsR1}
|{\tilde{f}}_m(x^*)-f_m(x^*)|\le ||{\tilde{\hat f}}_m-\hat f_m||_{\infty}\cdot\frac{1}{\beta}(\frac{k+2}{k})^{\eta_{k}}\frac{k}{\eta_k-1}[1+C_3(\frac{C_{LY}}{1-\gamma}+\frac{1}{|\Delta|})].
\end{equation}
Thus, to ensure $|{\tilde f}_m(x^*)-f_m(x^*)|\le\frac{R}{3}$, ${\tilde{\hat f}}_m$ should be computed to a precision that satisfies  
$$ ||{\tilde{\hat f}}_m-\hat f_m||_{\infty}\le\frac{1}{\frac{1}{\beta}(\frac{k+2}{k})^{\eta_{k}}\frac{k}{\eta_k-1}[1+C_3(\frac{C_{LY}}{1-\gamma}+\frac{1}{|\Delta|})]}\cdot\frac{R}{3}.$$
\item For (6), as in (5), we also suppose $x^\ast\in W_k$, and $T_1(x)=x+2^{\alpha}x^{1+\alpha}$. Firstly, by \eqref{ineq1}, each term in the sum 
\begin{equation}\label{wsum}
\sum\limits_{n=1}\limits^{\infty} \left(\frac{{\tilde{\hat{f}}_m}(T_{2}^{-1}T_{1}^{-(n-1)}x^*)}{|DT^{(n)}(T_{2}^{-1}T_{1}^{-(n-1)}x^*)|}\right)
\end{equation}
is bounded above by 
$$\sup_{y\in \Delta}{\tilde{\hat f}}_{m}(y)\cdot\frac{1}{\beta}(\frac{k+2}{n+k})^{\eta_k}.$$ 
Since $\eta_k>1,$ the tail of the sum, starting from $N_1^*+1$, in \eqref{wsum} can be approximated as follows:

\bigskip

Choose $N_1^*$ such that
\begin{equation}\label{ineq2}
\sup_{y\in \Delta}{\tilde{\hat f}}_{m}(y)\cdot\frac{{(N_1^\ast+k)}^{-\eta_k+1}}{\eta_k-1}\cdot(k+2)^{\eta_k}\le \frac{R}{6}.
\end{equation}
Secondly, using $\sum_{k=1}^{\infty}k{\tilde{\hat\mu}}_m(Z_k)\ge 1$, and recalling that $\alpha\in(0,1)$ and $\inf_{y\in\Delta}{\tilde{\hat f}}_m(y)>0$, we have
\begin{equation}\label{thecs}
\begin{split}
|{\tilde c}_{\tau,m}- {\tilde c}_{\tau,m}(N_2^\ast)|&=\frac{\sum\limits_{N_2^\ast+1}^{\infty}k{\tilde{\hat\mu}}_m(Z_k)}{\sum_{k=1}^{\infty}k{\tilde{\hat\mu}}_m(Z_k)\cdot\sum_{k=1}^{N_2^\ast}k{\tilde{\hat\mu}}_m(Z_k)}\le\sup_{y\in\Delta}{\tilde{\hat f}}_m\frac{\sum\limits_{N_2^\ast+1}^{\infty}k\hat\lambda(Z_k)}{\sum_{k=1}^{N_2^\ast}k{\tilde{\hat\mu}}_m(Z_k)}\\
&\le \sup_{y\in\Delta}{\tilde{\hat f}}_m\frac{\sum\limits_{N_2^\ast+1}^{\infty}\frac{C_2}{\beta}\cdot k\cdot k^{-1-\frac{1}{\alpha}}}{\inf_{y\in\Delta}{\tilde{\hat f}}_m\cdot\hat\lambda(Z_1)}\le\frac{\sup_{y\in\Delta}{\tilde{\hat f}}_m}{\inf_{y\in\Delta}{\tilde{\hat f}}_m\cdot\hat\lambda(Z_1)} \frac{C_2}{\beta}\cdot \frac{{N_2^*}^{\frac{-1}{\alpha}+1}}{\frac{1}{\alpha}-1}.
\end{split}
\end{equation}
Thirdly, using \eqref{thesum} and \eqref{thecs}, we have
\begin{equation}\label{2ndN}
\begin{split}
|{\tilde c}_{\tau,m}- {\tilde c}_{\tau,m}(N_2^\ast)|{\tilde f}_{m,N_2^*}(x^*)\le &\left(\frac{\sup_{y\in\Delta}{\tilde{\hat f}}_m}{\inf_{y\in\Delta}{\tilde{\hat f}}_m\cdot\hat\lambda(Z_1)} \frac{C_2}{\beta}\cdot \frac{{N_2^*}^{\frac{-1}{\alpha}+1}}{\frac{1}{\alpha}-1}\right){\tilde f}_{m}(x^*)\\
&\le \frac{\sup_{y\in\Delta}{\tilde{\hat f}}_m^2}{\inf_{y\in\Delta}{\tilde{\hat f}}_m\cdot\hat\lambda(Z_1)} \frac{C_2}{\beta^2}\cdot \frac{{N_2^*}^{\frac{-1}{\alpha}+1}}{\frac{1}{\alpha}-1}(\frac{k+2}{k})^{\eta_{k}}\frac{k}{\eta_k-1}.
\end{split}
\end{equation}
Choose\footnote{In our analysis we verified many items of the Algorithm analytically. Many of these steps maybe verified iteratively using the computer. For instance, $N_2^*$ maybe found iteratively using a computer. Simply keep $\sum_{k=1}^{N_2^\ast}k{\tilde{\hat\mu}}_m(Z_k)$ in the denominator on the right hand side of \eqref{2ndN} instead of replacing it by $\inf_{y\in\Delta}{\tilde{\hat f}}_m\cdot\hat\lambda(Z_1)$.  The iterative method may find a smaller value of $N_2^*$ that achieves the job.} $N^*_2$ such that the right hand side of \eqref{2ndN} is not greater than $\frac{R}{6}$. Finally, choose $N^*=\max\{N^*_1, N^*_2\}$. This will lead to the desired estimate $|{\tilde f}_m(x^*)-{\tilde f}_{m,N^*}(x^*)|\le \frac{R}{3}$.
\item for (7) using the above six steps, we obtain 
$$|f^*(x^*)-{\tilde f}_{m,N^*}(x^*)|\le |f^*(x^*)-f_m(x^*)|+|f_m(x^*)-{\tilde f}_{m}(x^*)|+|{\tilde f}_{m}(x^*)-{\tilde f}_{m,N^*}(x^*)|\leq R;$$
i.e., the {\bf{finite sum}} ${\tilde f}_{m,N^*}(x^*):={\tilde c}_{\tau,m}(N^\ast)\sum\limits_{n=1}\limits^{N^\ast} \left(\frac{{\tilde{\hat f}}_{m}(T_{2}^{-1}T_{1}^{-(n-1)}x^\ast)}{|DT^{(n)}(T_{2}^{-1}T_{1}^{-(n-1)}x^\ast)|}\right)$
is a rigorous approximation of $f^\ast(x^\ast)$ up to the pre-specified error $R.$
\end{itemize}

\bigskip

\noindent{\bf Acknowledgment:} The authors would like to thank an anonymous referee for comments that improved both the presentation and the content of the paper.

\bibliographystyle{amsplain}

\end{document}